\documentclass[a4paper]{amsart}

\usepackage{amsmath}
\usepackage{amssymb}
\usepackage{bm} 
\usepackage{mathrsfs} 
\usepackage{tensor} 
\usepackage{braket} 
\usepackage[italicdiff]{physics} 

\DeclareMathOperator{\supp}{supp}

\newcommand{\clop}[2]{[#1, #2)}
\newcommand{\opcl}[2]{(#1, #2]}

\usepackage{amsthm} 

\usepackage{cleveref} 

\usepackage{autonum} 

\theoremstyle{plain}
\newtheorem{theorem}{Theorem}[section]
\newtheorem{lemma}[theorem]{Lemma}
\newtheorem{proposition}[theorem]{Proposition}
\newtheorem{corollary}[theorem]{Corollary}

\theoremstyle{definition}

\theoremstyle{remark}
\newtheorem{remark}[theorem]{Remark}

\crefname{theorem}{Theorem}{Theorems}
\crefname{lemma}{Lemma}{Lemmas}
\crefname{proposition}{Proposition}{Propositions}
\crefname{corollary}{Corollary}{Corollaries}
\crefname{claim}{Claim}{Claims}
\crefname{assumption}{Assumption}{Assumptions}
\crefname{definition}{Definition}{Definitions}
\crefname{remark}{Remark}{Remarks}
\crefname{example}{Example}{Examples}
\crefname{equation}{}{}
\crefname{chapter}{Chapter}{Chapters}
\crefname{section}{Section}{Sections}
\crefname{subsection}{Section}{Sections}
\crefname{appendix}{Appendix}{Appendices}

\crefname{condition}{}{}

\creflabelformat{equation}{(#2#1#3)}
\creflabelformat{condition}{(#2#1#3)}

\numberwithin{equation}{section}

\usepackage{comment}
\usepackage{tikz-cd}

\usepackage[non-sorted-cites,initials,alphabetic,nobysame]{amsrefs}

\AtBeginDocument{%
	\def\MR#1{}
}

\title{Stability of the existence of a pseudo-Einstein contact form}
\author{Yuya Takeuchi}
\address{Department of Mathematics \\ Graduate School of Science \\ Osaka University
	\\ 1-1 Machikaneyama-cho, Toyonaka, Osaka 560-0043, Japan}
\email{yu-takeuchi@cr.math.sci.osaka-u.ac.jp}



\begin{document}

\begin{abstract}
	A pseudo-Einstein contact form plays a crucial role
	in defining some global invariants of closed strictly pseudoconvex CR manifolds.
	In this paper,
	we prove that
	the existence of a pseudo-Einstein contact form is preserved
	under deformations as a real hypersurface in a fixed complex manifold
	of complex dimension at least three.
\end{abstract}

\maketitle

\section{Introduction} \label{section:introduction}

A pseudo-Einstein contact form,
which was first introduced by Lee~\cite{Lee88},
is necessary for defining some global CR invariants:
the total $Q$-prime curvature~\cites{Case-Yang13,Hirachi14}
and the boundary term of the renormalized Gauss-Bonnet-Chern formula~\cite{Marugame16}.
When we consider the variation of such an invariant,
the question arises
whether the existence of a pseudo-Einstein contact form
is preserved under deformations of a CR structure.
In this paper,
we will show this stability for deformations as a real hypersurface
in a fixed complex manifold of complex dimension at least three.
More precisely,
we will prove the following

\begin{theorem} \label{thm:existence-of-flat-Hermitian-metric}
	Let $\Omega$ be a relatively compact strictly pseudoconvex domain
	in a complex manifold $X$ of complex dimension at least three.
	Assume that its boundary $M = \partial \Omega$ admits a pseudo-Einstein contact form.
	Then there exists a neighborhood $U$ of $M$ in $X$ such that
	its canonical bundle has a flat Hermitian metric.
\end{theorem}

The stability for wiggles follows from this theorem
and a necessary and sufficient condition
to the existence of a pseudo-Einstein contact form (\cref{prop:equivalence-of-pE-forms}).

\begin{corollary} \label{cor:stability-of-existence-of-pE-form}
	Let $\Omega$, $X$, $M$, and $U$ be as in \cref{thm:existence-of-flat-Hermitian-metric}.
	Then any strictly pseudoconvex real hypersurface $M'$ in $U$
	admits a pseudo-Einstein contact form.
\end{corollary}

Note that
this stability may have been already known
when an ambient complex manifold is
a Stein manifold of dimension at least three;
see \cref{rem:pE-form-for-Stein-case}.

Here we give an outline of a proof of \cref{thm:existence-of-flat-Hermitian-metric}.
Take a tubular neighborhood $U$ of $M$ in $X$.
The existence of a pseudo-Einstein contact form on $M$ implies that 
there is a flat Hermitian metric on the canonical bundle of $U \cap \Omega$
if we take $U$ sufficiently small.
By using the Bott-Chern class,
we will show that $K_{U}$ admits a flat Hermitian metric
if the morphism
\begin{equation} \label{eq:morphism-between-cohomology-of-sheaf-of-hol-functions}
	H^{1}(U, \mathcal{O}) \to H^{1}(U \cap \Omega, \mathcal{O})
\end{equation}
induced by the inclusion is injective (\cref{lemma:iff-condition-for-flat-metric}).
On the other hand,
a result of Andreotti and Grauert~\cite{Andreotti-Grauert62}
yields that \cref{eq:morphism-between-cohomology-of-sheaf-of-hol-functions}
is an isomorphism;
here we use the assumption
that the complex dimension of $X$ is at least three.
A proof of this fact will be given in \cref{section:proof-of-theorem}.

Before the end of the introduction,
we remark a relation between our result and the Lee conjecture.
Lee~\cite{Lee88}*{Proposition D} has proved that
the first Chern class $c_{1}(T^{1, 0} M)$ of $T^{1, 0} M$ is equal to zero in $H^{2}(M, \mathbb{R})$
if $M$ admits a pseudo-Einstein contact form,
and conjectured that the converse also holds if $M$ is closed;
this is called the \emph{Lee conjecture}.
There are some affirmative results on this conjecture%
~\cites{Lee88,Dragomir94,Cao-Chang07,Chen-Saotome-Wu12,Chang-Chang-Tie14},
but it is still open.
(Remark that
we need an extra assumption on the pseudo-Hermitian torsion
in \cite{Chang-Chang-Tie14}*{Theorem 1.1},
which has been pointed out in the erratum \cite{Chang-chang-Tie16}.)
The stability of the existence of a pseudo-Einstein contact form
follows from the Lee conjecture since the first Chern class of a CR structure is invariant
under deformations of a CR structure.
In other words,
\cref{cor:stability-of-existence-of-pE-form} can be considered as
one of affirmative results on the Lee conjecture.

\section*{Acknowledgements}
This paper is a part of the author's Ph.D. thesis.
He is grateful to his supervisor Kengo Hirachi
for various valuable suggestions.
He also would like to thank Taiji Marugame for helpful comments on \cref{rem:pE-form-for-Stein-case}.
This work was partially supported by JSPS Research Fellowship for Young Scientists,
JSPS KAKENHI Grant Numbers JP16J04653 and JP19J00063,
and the Program for Leading Graduate Schools, MEXT, Japan.

\section{Preliminaries} \label{section:preliminaries}

We first recall some facts on strictly pseudoconvex domains;
see \cite{several-complex-variables-VII}*{Chapter V} and references therein for details.
In what follows,
the word ``domain'' means a relatively compact connected open set.
Let $\Omega$ be a domain with smooth boundary $M$
in an $(n + 1)$-dimensional complex manifold $X$.
A \emph{defining function} of $\Omega$
is a smooth function on $X$ such that
$\Omega = \rho^{-1}((- \infty, 0))$,
$M = \rho^{- 1}(0)$, and $d \rho \neq 0$ on $M$.
A domain $\Omega$ is said to be \emph{strictly pseudoconvex}
if we can take a defining function of $\Omega$ that is strictly plurisubharmonic near $M$.
It is known that any strictly pseudoconvex domain $\Omega$ is holomorphically convex,
and consequently,
there exist a Stein space $Z$
and a proper surjective holomorphic map $\varphi \colon \Omega \to Z$
having some good properties,
called the \emph{Remmert reduction} of $\Omega$.
In our setting,
$\varphi$ is described as follows.
A compact analytic subset $E$ of positive dimension at every point in $\Omega$
is called a \emph{maximal compact analytic subset} of $\Omega$
if it is maximal among such subsets with respect to inclusion relations;
this $E$ is determined uniquely by $\Omega$.
The map $\varphi$ contracts each connected component of $E$ to a point,
and induces a biholomorphism $\Omega \setminus E \to Z \setminus \varphi(E)$.
In particular, $Z$ has at most finite normal isolated singularities.

We next give a brief introduction to CR manifolds.
Let $M$ be a $(2 n + 1)$-dimensional manifold without boundary.
A \emph{CR structure} is an $n$-dimensional complex subbundle $T^{1, 0} M$
of the complexified tangent bundle $T M \otimes \mathbb{C}$
satisfying the following conditions:
\begin{equation}
	T^{1, 0}M \cap T^{0, 1}M = 0, \qquad
	[\Gamma(T^{1, 0} M), \Gamma(T^{1, 0} M)] \subset \Gamma(T^{1, 0} M),
\end{equation}
where $T^{0, 1} M$ is the complex conjugate of $T^{1, 0} M$ in $T M \otimes \mathbb{C}$.
A typical example of a CR manifold is a real hypersurface $M$ in a complex manifold $X$;
it has the natural CR structure
\begin{equation}
	T^{1, 0} M
	= T^{1, 0} X \cap (T M \otimes \mathbb{C}).
\end{equation}
A CR structure $T^{1, 0} M$ is said to be \emph{strictly pseudoconvex}
if there exists a nowhere-vanishing real one-form $\theta$ on $M$
such that
$\theta$ annihilates $T^{1, 0} M$ and
\begin{equation}
	- \sqrt{- 1} d \theta (Z, \overline{Z}) > 0, \qquad
	0 \neq Z \in T^{1, 0} M;
\end{equation}
we call such a one-form a \emph{contact form}.
Note that the boundary of a strictly pseudoconvex domain
is a strictly pseudoconvex CR manifold with respect to its natural CR structure.

It is known that there is a canonical one-to-one correspondence
between contact forms on $M$ and Hermitian metrics on the canonical bundle of $M$.
A contact form is said to be \emph{pseudo-Einstein}
if the corresponding Hermitian metric is flat;
see \cite{Hirachi-Marugame-Matsumoto17}*{Section 2.3} for details.
Remark that
this definition coincides with that given by Lee~\cite{Lee88}
if $n \geq 2$.
In this paper,
however,
we do not use this definition
but the following necessary and sufficient condition to the existence of a pseudo-Einstein contact form
in terms of a Hermitian metric on the canonical bundle of the ambient complex manifold.

\begin{proposition}[\cite{Hirachi-Marugame-Matsumoto17}*{Proposition 2.6}] \label{prop:equivalence-of-pE-forms}
	Let $M$ be a strictly pseudoconvex real hypersurface in a complex manifold $X$.
	Then $M$ admits a pseudo-Einstein contact form
	if and only if the canonical bundle $K_{X}$ of $X$ has a Hermitian metric
	that is flat on the pseudoconvex side near $M$.
\end{proposition}

This proposition implies that
any strictly pseudoconvex real hypersurface in a complex manifold $X$
admits a pseudo-Einstein contact form
if $K_{X}$ has a flat Hermitian metric.
Thus we can derive \cref{cor:stability-of-existence-of-pE-form}
from \cref{thm:existence-of-flat-Hermitian-metric}.
In the remainder of this paper,
we will prove \cref{thm:existence-of-flat-Hermitian-metric}.

\section{Bott-Chern class and the existence of a flat Hermitian metric} \label{section:Bott-Chern-class}

Let $X$ be a complex manifold.
The \emph{real Bott-Chern cohomology $H^{1, 1}_{BC}(X, \mathbb{R})$ of bi-degree $(1, 1)$}
is defined by
\begin{equation}
	H^{1, 1}_{BC}(X, \mathbb{R})
	= \Set{ \text{$d$-closed real $(1, 1)$-forms on $X$} }
	/ \Set{ \sqrt{-1} \partial \overline{\partial} \psi \mid \psi \in C^{\infty}(X, \mathbb{R})}.
\end{equation}
Let $f \colon Y \to X$ be a holomorphic map between complex manifolds.
Then it defines the natural morphism
$f^{*} \colon H^{1, 1}_{BC}(X, \mathbb{R}) \to H^{1, 1}_{BC}(Y, \mathbb{R})$
induced by the pullback of $(1, 1)$-forms.

For a holomorphic line bundle $L$ over $X$,
the \emph{first Bott-Chern class} $c_{1}^{BC}(L) \in H^{1, 1}_{BC}(X, \mathbb{R})$ is defined as follows.
Take a Hermitian metric $h$ of $L$.
Then the curvature $(\sqrt{-1} / 2 \pi) \Theta_{h} = - (\sqrt{-1} / 2 \pi) \partial \overline{\partial} \log h$
is a $d$-closed real $(1, 1)$-form on $X$,
and defines an element of $H^{1, 1}_{BC}(X, \mathbb{R})$.
This cohomology class is independent of the choice of $h$,
denoted by $c_{1}^{BC}(L)$.
From the definition,
$c_{1}^{BC}(L) = 0$ if and only if $L$ admits a flat Hermitian metric.
Note that $c_{1}^{BC}$ is natural;
that is,
$f^{*} c_{1}^{BC}(L) = c_{1}^{BC}(f^{*} L)$
for any holomorphic map $f \colon Y \to X$.

The cohomology $H^{1, 1}_{BC}(X, \mathbb{R})$ has also a sheaf-theoretic interpretation.
Let $\mathcal{A}^{p, q}$ be the sheaf of smooth $(p, q)$-forms
and $\mathcal{P}$ be that of pluriharmonic functions.
Then there exists the following exact sequence of sheaves~\cite{Bigolin69}*{Teorema (2.1)}:
\begin{equation}
	0 \to
		\mathcal{P} \to
		\mathcal{A}^{0, 0}_{\mathbb{R}} \xrightarrow{\sqrt{-1} \partial \overline{\partial}}
		\mathcal{A}^{1, 1}_{\mathbb{R}} \xrightarrow{d}
		(\mathcal{A}^{2, 1} \oplus \mathcal{A}^{1, 2})_{\mathbb{R}}.
\end{equation}
Here the subscript $\mathbb{R}$ means the subsheaf consisting of real forms.
This exact sequence implies that
$H^{1}(X, \mathcal{P})$ is isomorphic to $H^{1, 1}_{BC}(X, \mathbb{R})$.
Note that a holomorphic map $f \colon Y \to X$ induces a natural morphism
$f^{*} \colon H^{1}(X, \mathcal{P}) \to H^{1}(Y, \mathcal{P})$,
which is compatible with $f^{*} \colon H^{1, 1}_{BC}(X, \mathbb{R}) \to H^{1, 1}_{BC}(Y, \mathbb{R})$
defined above.

This formulation gives a sufficient condition to the existence of a flat Hermitian metric.

\begin{lemma} \label{lemma:iff-condition-for-flat-metric}
	Let $X$ and $Y$ be complex manifolds
	and $f \colon Y \to X$ be a holomorphic map.
	Assume that $f$ induces injective morphisms
	$H^{1}(X, \mathcal{O}) \hookrightarrow H^{1}(Y, \mathcal{O})$
	and $H^{2}(X, \mathbb{R}) \hookrightarrow H^{2}(Y, \mathbb{R})$,
	and a surjective morphism $H^{1}(X, \mathbb{R}) \twoheadrightarrow H^{1}(Y, \mathbb{R})$.
	Then,
	for any holomorphic line bundle $L$ over $X$,
	it admits a flat Hermitian metric
	if so does $f^{*} L$.
\end{lemma}

\begin{proof}
	Assume that $f^{*} L$ has a flat Hermitian metric.
	As we noted above,
	this is equivalent to $f^{*} c_{1}^{BC}(L) = c_{1}^{BC}(f^{*} L) = 0$.
	Hence it is enough to prove the injectivity
	of $f^{*} \colon H^{1}(X, \mathcal{P}) \to H^{1}(Y, \mathcal{P})$.
	Consider the following exact sequence of sheaves:
	\begin{equation}
		0 \to
		\mathbb{R} \xrightarrow{\sqrt{- 1}}
		\mathcal{O} \xrightarrow{\Re}
		\mathcal{P} \to
		0.
	\end{equation}
	This induces the following commutative diagram with exact rows:
	\begin{equation}
		\begin{tikzcd}
			H^{1}(X, \mathbb{R}) \arrow[d] \arrow[r]
				& H^{1}(X, \mathcal{O}) \arrow[d] \arrow[r]
				& H^{1}(X, \mathcal{P}) \arrow[d] \arrow[r]
				& H^{2}(X, \mathbb{R}) \arrow[d] \\
			H^{1}(Y, \mathbb{R}) \arrow[r]
				& H^{1}(Y, \mathcal{O}) \arrow[r]
				& H^{1}(Y, \mathcal{P}) \arrow[r]
				& H^{2}(Y, \mathbb{R}).
		\end{tikzcd}
	\end{equation}
	The injectivity of $H^{1}(X, \mathcal{P}) \to H^{1}(Y, \mathcal{P})$
	follows from an easy diagram chasing.
\end{proof}

\section{Proof of \cref{thm:existence-of-flat-Hermitian-metric}} \label{section:proof-of-theorem}

Let $X$, $\Omega$, and $M$ be as in \cref{thm:existence-of-flat-Hermitian-metric}.
We first reduce the problem on $X$ to that on a Stein space.
Take a defining function $\rho$ of $\Omega$
that is strictly plurisubharmonic near the boundary.
Without loss of generality,
we may assume that $\rho \colon X \to \mathbb{R}$ is proper.
Then,
for sufficiently small $\delta > 0$,
there exists a diffeomorphism
\begin{equation}
	\chi \colon (- 2 \delta, 2 \delta) \times M \to \rho^{- 1}((- 2 \delta, 2 \delta))
\end{equation}
such that $\chi(0, p) = p$ and $\rho(\chi(t, p)) = t$.
Replacing $\delta$ to a smaller one if necessary,
we may assume that $\rho$ is strictly plurisubharmonic on $\rho^{- 1}((- 2 \delta, 2 \delta))$.
In particular,
$\Omega' = \rho^{- 1}((- \infty, \delta))$ is a strictly pseudoconvex domain in $X$ containing $\Omega$.
Consider the Remmert reduction $\varphi \colon \Omega' \to Z$.
From the strict plurisubharmonicity of $\rho$,
it follows that the maximal compact analytic subset of $\Omega'$
cannot intersect with $\rho^{- 1}((- \delta, \delta))$;
in particular,
$\varphi$ is a biholomorphism on $\rho^{- 1}((- \delta, \delta))$.
Without loss of generality,
we may assume that $\rho$ descends to a smooth function $Z \to \mathbb{R}$;
use the same letter $\rho$ for abbreviation.
It is sufficient to show the existence
of a neighborhood $U \subset \rho^{- 1}((- \delta, \delta))$ of $M = \rho^{- 1}(0)$
such that $K_{U}$ has a flat Hermitian metric.
To this end,
we need to construct a ``good'' exhaustion function on $Z$.

\begin{lemma} \label{lem:good-exhaustion}
	
	Fix $0 < \alpha < \delta$.
	There exists a smooth non-negative strictly plurisubharmonic exhaustion function $\phi$ on $Z$
	satisfying the following conditions:
	\begin{itemize}
		\item $\phi^{- 1}(0)$ coincides with the singular set $A$ of $Z$;
		\item $\phi$ is of the form
			\begin{equation}
				\phi(p)
				= \frac{\rho(p)}{\delta (\delta - \rho(p))} + K
			\end{equation}
			on $\rho^{- 1}((- \alpha, \delta))$ for a constant $K > 0$ ;
		\item $\phi < K$ on $\rho^{- 1}(\opcl{- \infty}{- \alpha})$.
	\end{itemize}
	
\end{lemma}

The proof of this lemma is slightly complicated,
and so will be given later.
Now,
we complete the proof of \cref{thm:existence-of-flat-Hermitian-metric}
using \cref{lemma:iff-condition-for-flat-metric,lem:good-exhaustion}.
Note that our proof is similar in spirit
to the proof of \cite{Yau81}*{Theorem B}.

\begin{proof}[Proof of \cref{thm:existence-of-flat-Hermitian-metric}]
	Set
	\begin{equation}
		\Omega(a, b) = \Set{ K + a < \phi < K + b}
	\end{equation}
	for $- \infty \leq a < b$.
	Note that $\phi^{- 1}(K) = M$ and $\Omega(- K, b) = \Omega(- \infty, b) \setminus A$.
	It is enough to prove that
	the canonical bundle of $\Omega(- \epsilon, \epsilon)$ admits a flat Hermitian metric
	for some $\epsilon > 0$
	if $M$ has a pseudo-Einstein contact form.
	The existence of a pseudo-Einstein contact form on $M$ implies that
	the canonical bundle of $\Omega(- \epsilon, 0)$ has a flat Hermitian metric
	for sufficiently small $\epsilon > 0$  by \cref{prop:equivalence-of-pE-forms}.
	We may also assume,
	by making $\epsilon$ small if necessary,
	the inclusion $\Omega(- \epsilon, 0) \hookrightarrow \Omega(- \epsilon, \epsilon)$
	induces isomorphisms
	\begin{align}
		H^{1}(\Omega(- \epsilon, \epsilon), \mathbb{R})
			& \xrightarrow{\simeq} H^{1}(\Omega(- \epsilon, 0), \mathbb{R}), \\
		H^{2}(\Omega(- \epsilon, \epsilon), \mathbb{R})
			& \xrightarrow{\simeq} H^{2}(\Omega(- \epsilon, 0), \mathbb{R}).
	\end{align}
	According to \cref{lemma:iff-condition-for-flat-metric},
	it suffices to prove that
	\begin{equation}
		H^{1}(\Omega(- \epsilon, \epsilon), \mathcal{O})
			\to H^{1}(\Omega(- \epsilon, 0), \mathcal{O})
	\end{equation}
	is also an isomorphism.
	Consider the following commutative diagram induced by inclusions:
	\begin{equation}
		\begin{tikzcd}
			H^{1}(\Omega(- K, \epsilon), \mathcal{O}) \arrow[d] \arrow[r,]
				& H^{1}(\Omega(- \epsilon, \epsilon), \mathcal{O}) \arrow[d] \\
			H^{1}(\Omega(- K, 0), \mathcal{O}) \arrow[r,]
				& H^{1}(\Omega(- \epsilon, 0), \mathcal{O}) .
		\end{tikzcd}
	\end{equation}
	From~\cite{Andreotti-Grauert62}*{Th\'{e}or\`{e}me 15},
	it follows that each row is an isomorphism;
	here we use the assumption
	that the complex dimension of $X$ is at least three.
	Hence it is sufficient to show the left column is an isomorphism.
	Since $\Omega(- \infty, \epsilon)$ and $\Omega(- \infty, 0)$ are Stein spaces,
	we obtain the following commutative diagram whose rows are isomorphisms:
	\begin{equation}
		\begin{tikzcd}
			 H^{1}(\Omega(- K, \epsilon), \mathcal{O}) \arrow[d] \arrow[r, "\simeq"]
				& H^{2}_{A}(\Omega(- \infty, \epsilon), \mathcal{O}) \arrow[d] \\
			 H^{1}(\Omega(- K, 0), \mathcal{O}) \arrow[r, "\simeq"]
				& H^{2}_{A}(\Omega(- \infty, 0), \mathcal{O}) .
		\end{tikzcd}
	\end{equation}
	On the other hand,
	the right column of the above diagram is also an isomorphism
	by the excision property of the local cohomology.
	This completes the proof.
\end{proof}

What is left is to show \cref{lem:good-exhaustion},
the existence of a ``good'' exhaustion function $\phi$ on $Z$.

\begin{proof}[Proof of \cref{lem:good-exhaustion}]
	
	As noted in \cref{section:preliminaries},
	the singular set $A$ of $Z$ is finite,
	given by $A = \Set{p_{1}, \dots , p_{k}} \subset Z$.
	We first construct a smooth non-negative strictly plurisubharmonic exhaustion function $\psi$
	on $Z$ with $\psi^{- 1}(0) = A$.
	There exists a proper holomorphic regular embedding $f \colon Z \to \mathbb{C}^{N}$
	for sufficiently large $N$~\cite{Narasimhan60}*{Theorem 6};
	in what follows,
	we identify $Z$ with the image of $f$.
	Then
	$\psi_{0} = |z|^{2} + \sum_{j = 1}^{k} \log |z - p_{j}|^{2}$ is a strictly plurisubharmonic
	exhaustion function on $\mathbb{C}^{N}$ with $\psi_{0}^{- 1}(- \infty) = A$.
	Hence $\psi = \exp \psi_{0} = \exp (|z|^{2}) \prod_{j = 1}^{k} |z - p_{j}|^{2}$
	is a smooth non-negative strictly plurisubharmonic exhaustion function on $\mathbb{C}^{N}$
	with $\psi^{- 1}(0) = A$.

	Choose $\beta \in \mathbb{R}$ with $\alpha < \beta < \delta$,
	and take a smooth function $\lambda \colon \mathbb{R} \to [0, 1]$ on $\mathbb{R}$
	such that $\lambda \equiv 1$ on $(- \infty, - \beta)$
	and $\lambda \equiv 0$ on $(- \alpha, \infty)$.
	Then the function
	\begin{equation}
		\phi_{1} (p) = \lambda(\rho(p)) \psi(p)
	\end{equation}
	is strictly plurisubharmonic on $\rho^{- 1}((- \infty, - \beta))$
	and identically zero on $\rho^{- 1}((- \alpha, \delta))$.

	Next,
	take a non-negative smooth function $g_{1}$ on $\mathbb{R}$
	with
	\begin{equation}
		\supp g_{1} \subset ((2 \delta)^{- 1}, (\beta + \delta)^{- 1}),
		\qquad \int_{\mathbb{R}} g_{1} (t) d t = 1,
	\end{equation}
	and set
	\begin{equation}
		g_{2}(t)
		= \int_{0}^{t} \int_{0}^{s} g_{1}(r) d r d s.
	\end{equation}
	This $g_{2}$ is a non-negative and non-decreasing convex smooth function on $\mathbb{R}$,
	vanishes identically on $\opcl{- \infty}{(2 \delta)^{- 1}}$,
	and
	\begin{equation}
		g_{2}(t) = t - \delta^{- 1} + g_{2} \pqty{\delta^{- 1}} > 0
	\end{equation}
	on a neighborhood of $\clop{(\beta + \delta)^{- 1}}{\infty}$.
	The function
	\begin{equation}
		\phi_{2} (p)
		= g_{2} \pqty{\frac{1}{\delta - \rho(p)}}
	\end{equation}
	vanishes identically on $\rho^{- 1}(\opcl{- \infty}{- \delta})$,
	is plurisubharmonic on $\rho^{- 1}((- \delta, \delta))$,
	and
	\begin{equation}
		\phi_{2}(p)
		= \frac{\rho(p)}{\delta (\delta - \rho(p))} + g_{2} \pqty{\delta^{- 1}} > 0.
	\end{equation}
	on a neighborhood of $\rho^{- 1}(\clop{- \beta}{\delta})$.
	Hence,
	for any $\epsilon > 0$,
	the sum $\phi = \epsilon \phi_{1} + \phi_{2}$ is
	a non-negative smooth exhaustion function on $Z$
	such that it is strictly plurisubharmonic on
	$\rho^{- 1}((- \infty, - \beta) \cup (- \alpha, \delta))$,
	and satisfies $\phi^{- 1}(0) = A$.
	Since $\phi_{2}$ is strictly plurisubharmonic on the compact set $\rho^{- 1}([- \beta, - \alpha])$,
	the function $\phi$ is also strictly plurisubharmonic there for sufficiently small $\epsilon$.
	Replacing $\epsilon$ by a smaller one,
	we also have $\phi < g(\delta^{- 1})$ on $\rho^{- 1}(\opcl{- \infty}{- \alpha})$.
\end{proof}

\begin{remark} \label{rem:pE-form-for-Stein-case}
	Cao and Chang have stated that if $M$ is the boundary of a strictly pseudoconvex domain
	in a Stein manifold of complex dimension at least three,
	then $M$ admits a pseudo-Einstein contact form~\cite{Cao-Chang07}*{Main Theorem (2)}.
	However,
	as the author has pointed out in \cite{Takeuchi-Chern}*{Remark 4.3},
	there exists such an $M$
	satisfying $c_{1}(T^{1, 0} M) \neq 0$ in $H^{2}(M, \mathbb{R})$;
	in particular, $M$ has no pseudo-Einstein contact form.
	Here,
	we give a short proof of a corrected statement:
	``if $M$ is the boundary of a strictly pseudoconvex domain
	in a Stein manifold of complex dimension at least three,
	and satisfies $c_{1}(T^{1, 0} M) = 0$ in $H^{2}(M, \mathbb{R})$,
	then $M$ admits a pseudo-Einstein contact form''.
	A discussion in \cite{Lee88}*{Section 6} gives that 
	a closed strictly pseudoconvex CR manifold $(M, T^{1, 0} M)$
	of dimension greater than three
	admits a pseudo-Einstein contact form
	if $c_{1}(T^{1, 0} M) = 0$ in $H^{2}(M, \mathbb{R})$
	and the Kohn-Rossi cohomology $H^{0, 1}(M)$ of bi-degree $(0, 1)$ vanishes.
	On the other hand,
	a result of Yau~\cite{Yau81}*{Theorem B}
	yields that $H^{0, 1}(M) = 0$
	if $M$ is as in the statement.
	Hence $M$ admits a pseudo-Einstein contact form.
\end{remark}


\end{document}